\documentclass[11 pt,draft]{article}
\title{Abelian quotients of subgroups of the mapping class group and higher Prym representations}
\author{Andrew Putman\footnote{Supported in part by NSF grant DMS-1005318}\ \ and Ben Wieland}
\usepackage{amsmath}
\usepackage{amssymb}
\usepackage{amsthm}
\usepackage{epsfig}
\usepackage[vscale=0.75]{geometry}
\usepackage{amsfonts}
\usepackage{calc}
\usepackage{amscd}
\usepackage[font=small]{caption}
\usepackage{pinlabel}
\usepackage[all,cmtip]{xy}

\theoremstyle{plain}
\newtheorem{theorem}{Theorem}[section]
\newtheorem{maintheorem}{Theorem}

\newtheorem{lemma}[theorem]{Lemma}

\newtheorem{corollary}[theorem]{Corollary}
\newtheorem{conjecture}[theorem]{Conjecture}

\newtheorem{step}{Step}

\newcommand\BeginClaims{\setcounter{claim}{0}}

\newcommand\BeginSteps{\setcounter{step}{0}}

\theoremstyle{definition}
\newtheorem{claim}{Claim}
\newtheorem*{definition}{Definition}

\theoremstyle{remark}
\newtheorem*{remark}{Remark}

\DeclareMathOperator{\Hom}{Hom}

\DeclareMathOperator{\Mod}{Mod}

\DeclareMathOperator{\Sp}{Sp}

\DeclareMathOperator{\SL}{SL}

\newcommand\Curves{\ensuremath{\mathcal{C}}}
\newcommand\CNosep{\ensuremath{\mathcal{CNS}}}

\newcommand\Z{\ensuremath{\mathbb{Z}}}
\newcommand\Q{\ensuremath{\mathbb{Q}}}

\DeclareMathOperator{\HH}{H}
\DeclareMathOperator{\HHH}{\hat{H}}

\DeclareMathOperator{\Aut}{Aut}

\DeclareMathOperator{\Interior}{Int}

\newcommand\CaptionSpace{\hspace{0.2in}}

\newcommand\Set[2]{\ensuremath{\{\text{#1 $|$ #2}\}}}
\DeclareMathOperator{\End}{End}

\newcommand\Figure[3]{
\begin{figure}[t]
\centering
\centerline{\psfig{file=#2,scale=60}}
\caption{#3}
\label{#1}
\end{figure}}

\DeclareMathOperator{\Ivanov}{NVSZ}
\DeclareMathOperator{\Action}{NFO}

\begin{document}

\maketitle

\begin{abstract}
A well-known conjecture asserts that the mapping class group
of a surface (possibly with punctures/boundary) does not 
virtually surject onto $\Z$ if the genus of the surface is large.  
We prove that if this conjecture holds for some genus, then it 
also holds for all larger genera.  We also prove that if there
is a counterexample to this conjecture, then there must
be a counterexample of a particularly simple form.  We prove
these results by relating the conjecture to a family of
linear representations of the mapping class group that we
call the higher Prym representations.  They generalize
the classical symplectic representation.  
\end{abstract}

\section{Introduction}
\label{section:introduction}

Let $\Sigma_{g,n}^p$ be an orientable genus $g$ surface with $n$ boundary components and $p$ punctures
and let $\Mod_{g,n}^p$ be the {\em mapping class group} of $\Sigma_{g,n}^p$.  This is the group
of orientation-preserving homeomorphisms of $\Sigma_{g,n}^p$ that fix the boundary and punctures pointwise
modulo isotopies that fix the boundary and punctures pointwise.  The group $\Mod_{g,n}^p$ plays
an important role in areas ranging from low-dimensional topology to algebraic geometry.  
See \cite{FarbMargalitBook} for a survey.

There is a fruitful analogy between the mapping class group and lattices
in higher rank Lie groups.  Such lattices satisfy Kazhdan's property (T), and as a consequence
they do not virtually surject onto $\Z$.  In other words, if $\Gamma$ is a finite-index subgroup
of a lattice in a higher-rank Lie group, then $\HH_1(\Gamma;\Q) = 0$ (see, e.g.,\ 
\cite[Theorems 7.1.4 and 7.1.7]{ZimmerBook}).  The starting point for this paper
is the following well-known conjecture of Ivanov (see, e.g.,\ \cite[Problem 2.11.A]{KirbyList} and
\cite[\S 7]{IvanovProblems}).  It asserts that something similar happens for $\Mod_{g,n}^p$.  See
the end of this introduction for a summary of the previous literature on it.

\begin{conjecture}
\label{conjecture:ivanov}
For some $g \geq 3$ and $n,p \geq 0$, let $\Gamma < \Mod_{g,n}^p$ be a finite-index subgroup.  Then
$\HH_1(\Gamma;\Q) = 0$.
\end{conjecture}

\begin{remark}
J.\ Anderson \cite{AndersonT} recently proved that $\Mod_{g,n}^p$ does {\em not} satisfy Kazhdan's property (T) 
for $g \geq 2$.
\end{remark}

\begin{remark}
McCarthy \cite{McCarthyCofinite} and Taherkhani \cite{Taherkhani} 
proved that Conjecture \ref{conjecture:ivanov} fails for $g = 2$.  
\end{remark}

\paragraph{Stability theorem.}
We have two main theorems. The first 
will interest those who believe that Conjecture \ref{conjecture:ivanov} is true and the second will interest
those who do not.
Let $\Ivanov(g)$ stand for the assertion that Conjecture \ref{conjecture:ivanov} holds for $\Sigma_{g,n}^p$
for all $n,p \geq 0$.  Our first result is as follows.

\begin{maintheorem}[Stability]
\label{theorem:stability}
Assume that $\Ivanov(G)$ holds for some $G \geq 3$.  Then $\Ivanov(g)$ holds for all $g \geq G$.
\end{maintheorem}

\noindent
In particular, Theorem \ref{theorem:stability} says that to prove Conjecture \ref{conjecture:ivanov}, it
is enough to deal with the genus $3$ case.

\paragraph{Easy counterexamples.}
Our second main theorem says that if Conjecture \ref{conjecture:ivanov} fails, then there must
be a counterexample of a particularly simple form.  If $g \geq 2$ and $p \geq 1$, then
there is a well-known short exact sequence called the {\em Birman exact sequence} (see
\S \ref{section:birmanexactsequence}) that
takes the form
$$1 \longrightarrow \pi_1(\Sigma_{g,n}^{p-1}) \longrightarrow \Mod_{g,n}^p \longrightarrow \Mod_{g,n}^{p-1} \longrightarrow 1.$$
The terms in this exact sequence are as follows.
\begin{itemize}
\item The map  $\Mod_{g,n}^p \rightarrow \Mod_{g,n}^{p-1}$ comes from filling in a puncture $x$ on $\Sigma_{g,n}^p$.
\item $\pi_1(\Sigma_{g,n}^{p-1})$ is embedded in $\Mod_{g,n}^p$ as mapping classes the ``push'' 
the puncture $x$ around the surface $\Sigma_{g,n}^{p-1}$.  It is often known as the ``point-pushing'' subgroup of $\Mod_{g,n}^p$.
\end{itemize}
If $n \geq 1$, then this exact sequence splits, so we have a semidirect product decomposition
$$\Mod_{g,n}^p \cong \pi_1(\Sigma_{g,n}^{p-1}) \rtimes \Mod_{g,n}^{p-1}.$$
If $K < \pi_1(\Sigma_{g,n}^{p-1})$ and $G < \Mod_{g,n}^{p-1}$ are finite-index subgroups
such that $G$ normalizes $K$, then we obtain a finite-index subgroup
$$\Gamma = K \rtimes G < \Mod_{g,n}^p.$$
Observe now that
$$\HH_1(\Gamma;\Q) \cong \HH_1(G;\Q) \oplus (\HH_1(K;\Q))_G,$$
where $(\HH_1(K;\Q))_G$ are the {\em coinvariants} of the action of $G$ on $\HH_1(K;\Q)$, i.e.\
the quotient of $\HH_1(K;\Q)$ by the subspace generated by the set
$\Set{$x - g(x)$}{$x \in \HH_1(K;\Q)$, $g \in G$}$.  In particular, if 
$(\HH_1(K;\Q))_G \neq 0$, then $\HH_1(\Gamma;\Q) \neq 0$, in which case we will
say that $\Gamma$ surjects onto $\Z$ by the {\em finite-index point-pushing construction}.

\begin{maintheorem}[Easy counterexamples]
\label{theorem:counterexamples}
If $\Ivanov(g)$ fails, then there is a counterexample to $\Ivanov(g-1)$ that surjects onto $\Z$ by
the finite-index point-pushing construction.
\end{maintheorem}

\paragraph{Higher Prym representations.}
We prove Theorems \ref{theorem:stability} and \ref{theorem:counterexamples} by relating 
Conjecture \ref{conjecture:ivanov} to a family of linear representations
of $\Mod_{g,n}^p$ that we call the higher Prym representations.  As motivation, observe that
$\Mod_{g,n}^p$ acts on $\HH_1(\Sigma_{g,n}^p;\Z)$ and preserves the algebraic intersection
form.  If $n+p \leq 1$, then this is a nondegenerate alternating form, so in these cases
we get a representation $\Mod_{g,n}^p \rightarrow \Sp_{2g}(\Z)$.  This representation plays a fundamental
role in the study of $\Mod_{g,n}^p$.  

We generalize the symplectic representation of the mapping class group in the following way.  
Recall that a subgroup $G'$ of a group $G$ is {\em characteristic} if $f(G') \subset G'$ for all $f \in \Aut(G)$.  
Fix a basepoint $v_0 \in \Sigma_{g,n}^p$.  Regarding the basepoint $v_0$ as a puncture,
the group $\Mod_{g,n}^{p+1}$ acts on $\pi_1(\Sigma_{g,n}^p,v_0)$.  If $K < \pi_1(\Sigma_{g,n}^{p},v_0)$
is a finite-index characteristic subgroup, then $\Mod_{g,n}^{p+1}$ acts on the finite-dimensional
vector space $\HH_1(K;\Q)$.  

If $S$ is the finite cover of
$\Sigma_{g,n}^{p}$ corresponding to $K$, then $\HH_1(K;\Q) \cong \HH_1(S;\Q)$.
The action of $\Mod_{g,n}^{p+1}$ on $\HH_1(K;\Q)$ arises from lifting mapping classes on $\Sigma_{g,n}^{p+1}$
to the cover $S$ (observe that this uses the fact that $\Mod_{g,n}^{p+1}$ fixes
the basepoint $v_0 \in \Sigma_{g,n}^{p}$ -- if there were no basepoint, then one could only lift mapping classes modulo
the action of the deck group).  The {\em boundary subspace} $B$ of $\HH_1(K;\Q)$ is the subspace
spanned by the homology classes of the boundary components of $S$ and loops freely
homotopic into the punctures of $S$.  Define $V_K = \HH_1(K;\Q)/B$.  It is clear that $\Mod_{g,n}^{p+1}$ preserves $B$, so we obtain an
action of $\Mod_{g,n}^{p+1}$ on $V_K$.  We will call the resulting linear representation $\Mod_{g,n}^{p+1} \rightarrow \Aut(V_K)$
a {\em higher Prym representation}.

\begin{remark}
For $K < \pi_1(\Sigma_{g},v_0)$ such that $\pi_1(\Sigma_{g},v_0) / K$ is abelian,
the action of $\Mod_{g}^1$ on $\HH_1(K;\Q)$ was studied by Looijenga in \cite{LooijengaPrym} (though
he arranged the technical details a little differently).  He called
his representations Prym representations because they are related to Prym varieties.
\end{remark}

\begin{remark}
The vector space $V_K$ is the first homology group of the closed surface that results
from gluing discs to all the boundary components of $S$ and filling in all the punctures
of $S$.  This implies that $V_K$ has a nondegenerate intersection pairing that 
is preserved by $\Mod_{g,n}^{p+1}$, so the image of a higher Prym representation
lies in the symplectic group.
\end{remark}

\begin{remark}
Representations of $\Aut(F_n)$ similar to the higher Prym representations were studied
by Grunewald and Lubotzky \cite{GrunewaldLubotzky}.
\end{remark}

\paragraph{Nontriviality of the Prym representations.}
We make the following conjecture about the higher Prym representations.

\begin{conjecture}
\label{conjecture:actionconjecture}
Fix $g \geq 2$ and $n,p \geq 0$.  Let $K < \pi_1(\Sigma_{g,n}^{p})$ be a finite-index
characteristic subgroup.  Then for all
nonzero vectors $v \in V_K$, the $\Mod_{g,n}^{p+1}$-orbit of $v$ is infinite.
\end{conjecture}

\noindent
Equivalently, none of the higher Prym representations have subrepresentations that factor through
a finite group.  

\begin{remark}
When $n=p=0$ and $g \geq 3$ and $K < \pi_1(\Sigma_{g,n}^p)$ is such that $\pi_1(\Sigma_{g,n}^{p}) / K$ is abelian,
then Looijenga essentially determined the image $G < \Aut(V_K)$ of the higher Prym representation of $\Mod_{g,n}^{p+1}$.  
Letting $A = \pi_1(\Sigma_{g,n}^{p}) / K$, the group $A$ acts on $V_K$ and $G$ commutes with this action.  What Looijenga proved
is that $G$ is an arithmetic subgroup of the centralizer of $A$ in $\Aut(V_K)$.  Here we are regarding $V_K$ as a symplectic vector
space and $\Aut(V_K)$ as a symplectic group over $\Q$.  From this, it is not hard to show that Conjecture \ref{conjecture:actionconjecture} holds
in these cases.
\end{remark}

\begin{remark}
Conjecture \ref{conjecture:actionconjecture} is false for $g=0$ and $g=1$.  The case
$g=0$ appears (in a different language) in work of McMullen \cite{McMullenBraid}.  Translated
into our language, \cite[Theorem 8.1]{McMullenBraid} gives a list of finite-index subgroups
$K < \pi_1(\Sigma_{0,0}^p)$ such that the subgroups $\Gamma_K < \Mod_{0,0}^{p+1}$ preserving $K$
have nontrivial finite orbits in $V_K$.  Taking the intersection of all the $\Mod_{0,0}^{p+1}$-translates
of one of these $K$ gives a characteristic subgroup of $\pi_1(\Sigma_{0,0}^p)$ which is a counterexample to
Conjecture \ref{conjecture:actionconjecture}.  A counterexample to the case $g=1$
is discussed in Appendix \ref{appendix:counterexample}.
\end{remark}

\paragraph{Relation between conjectures.}
Conjectures \ref{conjecture:ivanov} and \ref{conjecture:actionconjecture} appear quite different.  However,
it turns out that they are essentially equivalent.  Let $\Ivanov(g,n,p)$ stand for the
assertion that Conjecture \ref{conjecture:ivanov} is true for $\Sigma_{g,n}^p$ and let
$\Action(g,n,p)$
stand for the assertion that Conjecture \ref{conjecture:actionconjecture} is true
for $\Sigma_{g,n}^p$.

\begin{maintheorem}
\label{theorem:main}
Fix $g \geq 3$ and $p \geq 0$.
\begin{itemize}
\item $\Action(g-1,n+1,p)$ implies $\Ivanov(g,n,p)$ for $n \geq 0$.
\item $\Ivanov(g,n,p+1)$ implies $\Action(g,n,p)$ for $n \geq 1$.
\end{itemize}
\end{maintheorem}

\begin{remark}
It is easy to see that $\Action(g,n,p)$ implies $\Action(g,n',p')$ for $n' \leq n$ and $p' \leq p$,
so Theorem \ref{theorem:main} also shows that $\Ivanov(g,1,p)$ 
implies $\Action(g,0,p)$ for $g \geq 3$ and $p \geq 0$.
\end{remark}

\begin{remark}
The proof of Theorem \ref{theorem:main} shows that to prove $\Ivanov(g,n,p)$, it is enough to
prove Conjecture \ref{conjecture:actionconjecture} for a cofinal set
of finite-index characteristic subgroups of $\pi_1(\Sigma_{g-1,n+1}^p)$.
\end{remark}

\paragraph{Derivation of Theorems \ref{theorem:stability} and \ref{theorem:counterexamples} from
Theorem \ref{theorem:main}.}
Theorem \ref{theorem:main} immediately implies Theorems \ref{theorem:stability} and 
\ref{theorem:counterexamples}.  First, we can apply Theorem \ref{theorem:main} twice
and see that if $g \geq 3$ and $n,p \geq 0$, then
$$\Ivanov(g,n+1,p+1) \Longrightarrow \Action(g,n+1,p) \Longrightarrow \Ivanov(g+1,n,p).$$
Iterating this, we see that if $\Ivanov(G,n+(g-G),p+(g-G))$ is true for some $G \geq 3$ and $g \geq G$ and
$n,p \geq 0$, then $\Ivanov(g,n,p)$ is true.  In other words, $\Ivanov(G)$ implies $\Ivanov(g)$ for
$g \geq G$.

As far as Theorem \ref{theorem:counterexamples} goes, we can again apply Theorem \ref{theorem:main}
twice and see that if $g \geq 3$ and $n,p \geq 0$, then
$$\lnot \Ivanov(g,n,p) \Longrightarrow \lnot \Action(g-1,n+1,p) \Longrightarrow \lnot \Ivanov(g-1,n+1,p+1).$$
The proof of the second implication here actually produces a counterexample to $\Ivanov(g-1,n+1,p+1)$ 
using the finite-index point-pushing construction.  The input for this construction is obtained
from a counterexample to the assertion $\Action(g-1,n+1,p)$.  See \S \ref{section:construction} for
more details.

\begin{remark}
An amusing property of the above derivation is that it is works whether Conjecture
\ref{conjecture:actionconjecture} holds or not.
\end{remark}

\paragraph{History.}
Fix some $g \geq 3$.  There are only a few known examples of finite-index subgroups $\Gamma < \Mod_{g,n}^p$ for
which $\HH_1(\Gamma;\Q)$ is known to vanish.  In \cite{HainTorelli}, Hain verifies this
for $\Gamma$ that contain the {\em Torelli group}, which is the kernel of the action
of $\Mod_{g,n}^p$ on $\HH_1(\Sigma_{g,0}^0;\Z)$.  This was later generalized to some deeper
subgroups by Boggi \cite{BoggiCompact} and the first author \cite{PutmanH1FiniteIndex}.

A related result, which was proven independently by the first author \cite{PutmanH1FiniteIndex}
and Bridson \cite{Bridson}, says that if $\Gamma < \Mod_{g,n}^p$ is a finite index subgroup and
$T_{\gamma} \in \Mod_{g,n}^p$ is a Dehn twist about a simple closed curve $\gamma$,
then the image of $T_{\gamma}^k$ in $\HH_1(\Gamma;\Q)$ vanishes for any $k$ such that $T_{\gamma}^k \in \Gamma$.  
This theorem will play an important role in our proof of Theorem \ref{theorem:main}.

We finally should mention Boggi's recent work on the congruence subgroup problem for $\Mod_{g,n}^p$ (see
\cite{BoggiCongruence}).  The congruence subgroup problem gives a conjectural classification of
all finite-index subgroups of $\Mod_{g,n}^p$.  Though Boggi's proof of the congruence
subgroup problem itself appears to be fatally flawed,
he does give a correct proof of the following beautiful result.  Let $\Curves_{g,n}^p$ be the
curve complex for $\Sigma_{g,n}^p$ (see \S \ref{section:curvecomplex} for details).  A theorem
of Harer says that $\HH_k(\Curves_{g,n}^p;\Z)=0$ for $k$ in some range.
Boggi proved that we also have $\HH_k(\Curves_{g,n}^p/\Gamma;\Q)=0$ for $k$ in this same range 
for $\Gamma$ a finite-index subgroup of $\Mod_{g,n}^p$.  This
result, which was proven using
the theory of weights on the cohomology of algebraic varieties
\cite{DeligneWeight},
%Hodge theory,
plays a fundamental role in our proof of Theorem
\ref{theorem:main}.

\paragraph{Outline.}
We begin in \S \ref{section:preliminaries} with a discussion of some preliminary material about
mapping class groups and group homology.  Next, in \S \ref{section:bigdirection} we prove
the portion of Theorem \ref{theorem:main} asserting that Conjecture \ref{conjecture:actionconjecture}
implies Conjecture \ref{conjecture:ivanov} (see Theorem \ref{theorem:actionimpliesivanov}).  Finally,
in \S \ref{section:construction} we prove the portion of Theorem \ref{theorem:main} asserting that
Conjecture \ref{conjecture:ivanov} implies Conjecture \ref{conjecture:actionconjecture}.

\section{Preliminaries}
\label{section:preliminaries}

This section has two parts.  Some facts about group homology are reviewed in
\S \ref{section:grouphomology} and some background about the mapping class group
is discussed in \S \ref{section:surfaces}.

\subsection{Group homology}
\label{section:grouphomology}

A good reference for this material is \cite{BrownCohomology}.  We begin with some notation.

\begin{definition}
Let $G$ be a group and let $M$ be a vector space upon which $G$ acts.
\begin{itemize}
\item The {\em invariants} of $G$ acting on $M$ are $M^G = \Set{$x \in m$}{$g(x)=x$ for all $g \in G$}$.
\item The {\em coinvariants} of $G$ acting on $M$ are $M_G = M/K$, where $K < M$ is the subspace spanned
by the set $\Set{$x-g(x)$}{$x \in M$, $g \in G$}$.
\end{itemize}
\end{definition}

\noindent
The invariants and coinvariants are related by the following lemma, whose proof is an easy exercise.

\begin{lemma}
\label{lemma:selfdualcoinv}
Let $G$ be a group and $M$ be a $G$-vector space.  Let $M^{\ast}$ denote the dual of $M$.
Then $(M_G)^{\ast} \cong (M^{\ast})^G$.
\end{lemma}

The next lemma we need is the following, which is a standard consequence
of the existence of the {\em transfer map} (see, e.g., \cite[Chapter III.9]{BrownCohomology}).

\begin{lemma}
\label{lemma:transfer}
If $G_1$ is a finite index subgroup of $G_2$,
then the map $\HH_1(G_1;\Q) \rightarrow \HH_1(G_2;\Q)$
is a surjection.
\end{lemma}

\noindent
If $G_1$ is a normal subgroup of $G_2$, then $G_2$ acts on $G_1$ and we have the following
more precise lemma, which is an immediate consequence of the Hochschild-Serre spectral sequence.

\begin{lemma}
\label{lemma:supertransfer}
If $G_1$ is a finite index normal subgroup of $G_2$,
then we have $\HH_1(G_2;\Q) \cong (\HH_1(G_1;\Q))_{G_2}$.
\end{lemma}

The final result we need gives a decomposition of $\HH_1(G)$ for $G$ acting on a space $X$.  We need
the following definition.

\begin{definition}
A group $G$ acts on a simplicial complex $X$ {\em without rotations} if for all simplices $s$ of $X$,
the stabilizer subgroup $G_s$ stabilizes $s$ pointwise.
\end{definition}

\noindent
The theorem we need is as follows.  It follows easily from the two spectral sequences given in
\cite[Chapter VII.7]{BrownCohomology}.

\begin{theorem}
\label{theorem:equivariant}
Let a group $G$ act on a connected simplicial complex $X$ without rotations.  Fix a ring $R$.
Assume that $\HH_1(X;R) = \HH_1(X/G;R) = 0$.  Then the natural map
$$\bigoplus_{v \in X^{(0)}} \HH_1(G_v;R) \longrightarrow \HH_1(G;R)$$
is a surjection.
\end{theorem}

\subsection{Surfaces and mapping class groups}
\label{section:surfaces}

A good reference for this material is \cite{FarbMargalitBook}, which contains
proofs of all statements for which we do not give proofs.

\subsubsection{Embedding one surface into another}
\label{section:surfacefunctor}

Assume that $g \geq 2$ and $n,p \geq 0$.  Let $S$ be a subsurface of 
$\Sigma_{g,n}^p$, and denote by $\Mod(S)$ the mapping class group
of $S$.  There is then an induced map $\phi : \Mod(S) \rightarrow \Mod_{g,n}^p$ obtained by extending
mapping classes by the identity over $\Sigma_{g,n}^p \setminus S$.  This map is usually injective.  However,
in the following situations it is not injective.
\begin{itemize}
\item If $\Sigma_{g,n}^p \setminus \Interior(S)$ is a punctured disc with boundary component $\beta$, 
then we have a short exact sequence
\begin{equation}
\label{eqn:boundarytopuncture}
1 \longrightarrow \Z \longrightarrow \Mod(S) \longrightarrow \Mod_{g,n}^p \longrightarrow 1.
\end{equation}
Here the kernel $\Z$ is generated by the Dehn twist $T_{\beta}$.
\item If $\Sigma_{g,n}^p \setminus \Interior(S)$ is an annulus with boundary components $x$ and $y$ and
both $x$ and $y$ lie in $S$, then we have a short exact sequence
\begin{equation}
\label{eqn:complementscc}
1 \longrightarrow \Z \longrightarrow \Mod(S) \longrightarrow (\Mod_{g,n}^p)_{\gamma} \longrightarrow 1.
\end{equation}
Here $\gamma$ is a simple closed curve forming the core of the annulus $\Sigma_{g,n}^p \setminus \Interior(S)$
and $(\Mod_{g,n}^p)_{\gamma}$ is the subgroup of $\Mod_{g,n}^p$ consisting of mapping classes
that fix the isotopy class of $\gamma$ (as an oriented curve -- elements of $(\Mod_{g,n}^p)_{\gamma}$
cannot reverse the orientation of $\gamma$).  Also, $\Z$ is generated by $T_x T_{y}^{-1}$.
\end{itemize}

\subsubsection{Surfaces with boundary and the Birman exact sequence}
\label{section:birmanexactsequence}

Assume that $g \geq 2$ and $n,p \geq 0$.  As was discussed in \S \ref{section:introduction},
we have the Birman exact sequence
$$1 \longrightarrow \pi_1(\Sigma_{g,n}^{p}) \longrightarrow \Mod_{g,n}^{p+1} \longrightarrow \Mod_{g,n}^p \longrightarrow 1,$$
where the map $\Mod_{g,n}^{p+1} \longrightarrow \Mod_{g,n}^p$ comes from filling in a puncture $x$ of $\Sigma_{g,n}^{p+1}$
and $\pi_1(\Sigma_{g,n}^p)$ is embedded in $\Mod_{g,n}^{p+1}$ as mapping classes
that ``push'' $x$ around loops in $\Sigma_{g,n}^p$.  If $n \geq 1$, then this exact sequence splits
via a map $\Mod_{g,n}^p \hookrightarrow \Mod_{g,n}^{p+1}$ induced by an embedding
$\Sigma_{g,n}^p \hookrightarrow \Sigma_{g,n}^{p+1}$ such that $\Sigma_{g,n}^{p+1} \setminus \Interior \Sigma_{g,n}^p$ is
homeomorphic to $\Sigma_{0,2}^1$.  This originally appeared in \cite{BirmanExactSequence}, and a suitable
textbook reference is \cite{FarbMargalitBook}.

\Figure{figure:birmandiagram}{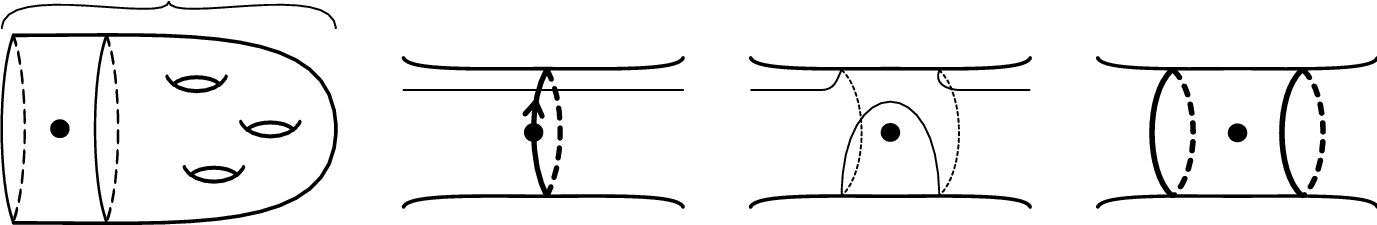}{a. We can split the Birman exact sequence
by a map $\Mod_{g,n}^p \hookrightarrow \Mod_{g,n}^{p+1}$ as shown.
\CaptionSpace b. A simple closed curve $\gamma \in \pi_1(\Sigma_{g,b}^p)$.
\CaptionSpace c. The result of pushing the basepoint around $\gamma$.
\CaptionSpace d. The associated element of $\Mod_{g,b}^{p+1}$ is $T_x T_y^{-1}$, where $x$
and $y$ are as shown.}

We will need the following standard lemma about the mapping classes associated to certain elements of $\pi_1(\Sigma_{g,n}^p)$.
A {\em multitwist} is a product $T_{\gamma_1}^{k_1} \cdots T_{\gamma_m}^{k_m}$, where
the $\gamma_i$ are disjoint simple closed curves and $k_i \in \Z$.

\begin{lemma}
\label{lemma:birmansimple}
Consider $\gamma \in \pi_1(\Sigma_{g,n}^p)$ that can be realized by a simple closed curve.
Then the element of $\Mod_{g,n}^{p+1}$ associated to $\gamma$ is a multitwist.
\end{lemma}

\noindent
For example, if $\gamma \in \pi_1(\Sigma_{g,n}^p)$ is as indicated in Figure \ref{figure:birmandiagram}.b, then as
indicated in Figure \ref{figure:birmandiagram}.c the mapping class associated to $\gamma$ equals $T_x T_{y}^{-1}$, where $x$
and $y$ are the curves indicated in Figure \ref{figure:birmandiagram}.d.

To help us recognize elements of $\pi_1(\Sigma_{g,n}^p)$ that can be realized by simple closed curves,
the following well-known lemma will be useful.

\begin{lemma}
\label{lemma:conjugatesimple}
Fix $g,n,p \geq 0$, and let $v_0 \in \Interior(\Sigma_{g,n}^p)$.  Consider $\gamma \in \pi_1(\Sigma_{g,n}^p,v_0)$
such that $\gamma$ is freely homotopic to a simple closed curve.  Then $\gamma$ can be realized by a
based simple closed curve.
\end{lemma}
\begin{proof}
The conditions imply that there is some $\gamma' \in \pi_1(\Sigma_{g,n}^p,v_0)$ that can be realized by
a based simple closed curve such that $\gamma'$ is conjugate to $\gamma$.  Puncturing $\Sigma_{g,n}^p$ at $v_0$,
the group $\Mod_{g,n}^{p+1}$ acts on $\pi_1(\Sigma_{g,n}^p,v_0)$.  The Dehn-Nielsen-Baer theorem (see, e.g.,
\cite{FarbMargalitBook}) 
implies that the image of the resulting map $\Mod_{g,n}^{p+1} \rightarrow \Aut(\pi_1(\Sigma_{g,n}^p,v_0))$
contains all inner automorphisms.  This implies that there is some $f \in \Mod_{g,n}^{p+1}$ such that
$f(\gamma') = \gamma$.  Since $\gamma'$ can be realized by a simple closed curve, we conclude that $\gamma$
can as well.
\end{proof}

\subsubsection{The curve complex}
\label{section:curvecomplex}

We will need some results about the curve complex.  This space, which 
was introduced by Harvey \cite{HarveyComplex}, plays a role in the study
of the mapping class group analogous to the role of the Tits building of an
algebraic group.  We start with the definition.

\begin{definition}
Fix $g,n,p \geq 0$.  A simple closed curve on $\Sigma_{g,n}^p$ is {\em nontrivial} if
it is not homotopic to a point, a puncture, or a boundary component.
The {\em curve complex} on $\Sigma_{g,n}^p$, denoted $\Curves_{g,n}^p$, is the simplicial complex whose
$k$-simplices are sets $\{\gamma_0,\ldots,\gamma_k\}$ of homotopy classes of nontrivial simple
closed curves on $\Sigma_{g,n}^p$ that can be realized disjointly.
\end{definition}

\noindent
We will also need the following space.

\begin{definition}
The {\em nonseparating curve complex},
denoted $\CNosep_{g,n}^p$, is the subcomplex of $\Curves_{g,n}^p$ consisting of simplices
$\{\gamma_0,\ldots,\gamma_k\}$ such that each $\gamma_i$ is the homotopy class of a nonseparating simple closed
curve.
\end{definition}

\noindent
For us, the key result about $\Curves_{g,n}^p$ is the following, which is due to Harer \cite{HarerStable}.

\begin{theorem}[\cite{HarerStable}]
\label{theorem:harercurvecpx}
Fix $g,n,p \geq 0$.  The space $\Curves_{g,n}^p$ is $(2g-2)$-connected if $n=p = 0$ and is
$(2g+n+p-3)$-connected if $n+p > 0$.
\end{theorem}

\noindent
We will also need the following folklore lemma (see, e.g., \cite[Lemma A.2]{PutmanCutPaste}), which
while not strictly about the curve complex is in the same spirit.

\begin{lemma}
\label{lemma:connectcurves}
Fix $g \geq 1$ and $n,p \geq 0$.  
Let $\gamma$ and $\gamma'$ be simple closed nonseparating curves on $\Sigma_{g,n}^p$.  
There then exists a sequence
$\eta_1,\ldots,\eta_k$ of simple closed nonseparating curves on $\Sigma_{g,n}^p$ 
such that $\gamma = \eta_1$, such that
$\gamma' = \eta_k$, and such that $\eta_i$ and $\eta_{i+1}$ intersect exactly once for $1 \leq i < k$.
\end{lemma}

\subsubsection{Some results on finite-index subgroups}

Here we collect some basic results on finite-index subgroups of the mapping class group.  We
start with the following example, which will play a small role in our proof.

\begin{definition}
The {\em level $L$ subgroup} of $\Mod_{g,n}^p(L)$, denoted $\Mod_{g,n}^p(L)$, is the kernel of the
action of $\Mod_{g,n}^p(L)$ on $\HH_1(\Sigma_{g,n}^p;\Z/L)$.
\end{definition}

\begin{remark}
There are other possible definitions of $\Mod_{g,n}^p(L)$ when $n+p \geq 2$, but the above suffices
for our purposes.
\end{remark}

Next, we will need the following theorem, which was proven independently by Bridson \cite{Bridson}
and the second author \cite{PutmanH1FiniteIndex}.

\begin{theorem}[{\cite{Bridson, PutmanH1FiniteIndex}}]
\label{theorem:dehntwistsvanish}
Fix $g \geq 3$ and $n,p \geq 0$.  Let $\Gamma$ be a finite-index subgroup of $\Mod_{g,n}^p$, let
$\gamma$ be a simple closed curve on $\Sigma_{g,n}^p$, and let $k \geq 1$ be such that
$T_{\gamma}^k \in \Gamma$.  Then the image of $T_{\gamma}^k$ in
$\HH_1(\Gamma;\Q)$ is zero.
\end{theorem}

\noindent
In fact, we will need the following small extension of Theorem \ref{theorem:dehntwistsvanish}.  

\begin{corollary}
\label{corollary:multitwistsvanish}
Fix $g \geq 3$ and $n,p \geq 0$.  Let $\Gamma$ be a finite-index subgroup of $\Mod_{g,n}^p$ and let
$M \in \Mod_{g,n}^p$ be a multitwist such that $M \in \Gamma$.  Then the image of $M$ in
$\HH_1(\Gamma;\Q)$ is zero.
\end{corollary}
\begin{proof}
If $M = T_{\gamma_1}^{k_1} \cdots T_{\gamma_m}^{k_m}$ is a multitwist, then for some $K \geq 1$ we
have $T_{\gamma_i}^{K k_i} \in \Gamma$ for $1 \leq i \leq m$.  Since the $T_{\gamma_i}$ commute, we have
$M^K = T_{\gamma_1}^{K k_1} \cdots T_{\gamma_m}^{K k_m}$.  Theorem \ref{theorem:dehntwistsvanish} says
that the image of $T_{\gamma_i}^{K k_i}$ in $\HH_1(\Gamma;\Q)$ vanishes for $1 \leq i \leq m$, so the image
of $M^K$ (and thus $M$) does as well.
\end{proof}

The final theorem we need is the following deep result of Boggi
(cf.~\cite[Lemma 2.6]{ArbarelloCornalba}).

\begin{theorem}[{\cite[Lemma 5.5]{BoggiCongruence}}]
\label{theorem:boggicurvecpx}
Let $\Gamma$ be a finite-index subgroup of $\Mod_{g,n}^p$.  Then $\HH_k(\Curves_{g,n}^p/\Gamma;\Q) = 0$
for $1 \leq k \leq 2g-2$ if $n+p=0$ and for $1 \leq k \leq 2g+n+m-3$ if $n+p > 0$.
\end{theorem}

\begin{remark}
The proof of the main theorem of \cite{BoggiCongruence} has a fatal flaw, but the proof
of Theorem \ref{theorem:boggicurvecpx} is correct.  See \cite{BoggiReview} for details.
\end{remark}

\begin{remark}
Harer \cite{HarerStable} also proved that $\CNosep_{g,n}^p$ is $(g-2)$-connected, and one
can extract from Boggi's paper \cite{BoggiCongruence} a proof that $\HH_k(\CNosep_{g,n}^p/\Gamma;\Q) = 0$
for $1 \leq k \leq g-2$ and $\Gamma$ a finite-index subgroup of $\Mod_{g,n}^p$.  Using
this, we could simplify our proof of Lemma \ref{lemma:curvestabsurject}; however, we have
instead chosen to avoid it by using an elementary topological argument.
\end{remark}

\section{The mapping class group (conditionally) does not virtually surject onto $\Z$}
\label{section:bigdirection}

In this section, we prove that Conjecture \ref{conjecture:actionconjecture} 
implies Conjecture \ref{conjecture:ivanov}.  The actual proof is contained in
\S \ref{section:theproof} (see Theorem \ref{theorem:actionimpliesivanov}).  
This is proceeded by \S \ref{section:capping} and
\S \ref{section:curvestab}, which prove two necessary lemmas.
The proof has some annoying features when $g=3$.  To avoid having to deal with
this case separately, we make the following definition.

\begin{definition}
Fix $g,n,p \geq 0$, and let $\Gamma$ be a subgroup of $\Mod_{g,n}^p$.  Let
$T(\Gamma)$ be the subgroup of $\Gamma$ generated by the set
$$\Set{$M$}{$M \in \Mod_{g,n}^p$ is a multitwist, $M \in \Gamma$}$$
and define $\HHH_1(\Gamma;\Q) = \HH_1(\Gamma/T(\Gamma);\Q)$.
\end{definition}

\begin{remark}
If $g \geq 3$ and $\Gamma$ is finite-index, then Corollary \ref{corollary:multitwistsvanish} 
implies that $\HHH_1(\Gamma;\Q) = \HH_1(\Gamma;\Q)$.
\end{remark}

\subsection{Filling in punctures}
\label{section:capping}

In this section, we prove the following lemma, which says that the portion of a finite-index
subgroup $\Gamma < \Mod_{g,n}^p$ living in the point-pushing subgroup (i.e.\ the kernel of the 
Birman exact sequence) goes to zero in $\HHH_1(\Gamma;\Q)$.

\begin{lemma}
\label{lemma:fillpuncture}
Fix $g \geq 2$ and $n \geq 0$ and $p \geq 1$.  
Assume that Conjecture \ref{conjecture:actionconjecture} holds for $\Sigma_{g,n}^{p-1}$.  
Let $\Gamma < \Mod_{g,n}^p$ be a finite-index subgroup.  Fix a puncture
of $\Sigma_{g,n}^p$ and thus via the Birman exact sequence a point-pushing subgroup 
$\pi_1(\Sigma_{g,n}^{p-1}) < \Mod_{g,n}^p$.
Then the map $\HH_1(\Gamma \cap \pi_1(\Sigma_{g,n}^{p-1});\Q) \rightarrow \HHH_1(\Gamma;\Q)$ is the zero map.
\end{lemma}
\begin{proof}
Let $K$ be a finite-index characteristic subgroup of $\pi_1(\Sigma_{g,n}^{p-1})$ such that
$K < \Gamma \cap \pi_1(\Sigma_{g,n}^{p-1})$.  For example, $K$ could be the intersection
of all index $[\pi_1(\Sigma_{g,n}^{p-1}):\Gamma \cap \pi_1(\Sigma_{g,n}^{p-1})]$ subgroups
of $\pi_1(\Sigma_{g,n}^{p-1})$.  Lemma \ref{lemma:transfer} says that
the map $\HH_1(K;\Q) \rightarrow \HH_1(\Gamma \cap \pi_1(\Sigma_{g,n}^{p-1};\Q)$ is surjective,
so it is enough to show that the map $\HH_1(K;\Q) \rightarrow \HHH_1(\Gamma;\Q)$ is the zero map.
This will have two steps.

\BeginSteps
\begin{step}
The map $\HH_1(K;\Q) \rightarrow \HHH_1(\Gamma;\Q)$ factors through $V_K$.  
\end{step}

Let $\rho : S \rightarrow \Sigma_{g,n}^{p-1}$ be the finite cover of $\Sigma_{g,n}^{p-1}$
corresponding to $K$ and let $B < \HH_1(K;\Q) \cong \HH_1(S;\Q)$ be the boundary subspace.  Thus
by definition $V_K = \HH_1(K;\Q)/B$, so we must show that $B$ goes to $0$ in
$\HHH_1(\Gamma;\Q)$.
To do this, we must be careful with basepoints.  Let $v$ be the basepoint in $\Sigma_{g,n}^{p-1}$ and 
let $v_S \in S$ be the lift of $v$ to $S$ such that the image of the map
$\rho_{\ast} : \pi_1(S,v_S) \rightarrow \pi_1(\Sigma_{g,n}^{p-1},v)$
is $K$.

Let $\delta_S$ be either a boundary component of $S$ or a simple closed curve in $S$ that
is freely homotopic to a puncture of $S$.  The subspace $B$ is generated by the
homology classes of such $\delta_S$.  Chasing through
the definitions, it is enough to find some $\eta_S \in \pi_1(S,v_S)$ with the following
two properties.
\begin{itemize}
\item $\eta_S$ is freely homotopic to $\delta_S$.
\item Let
$$\eta = \rho_{\ast}(\eta_S) \in K < \pi_1(\Sigma_{g,n}^{p-1},v).$$ 
We then want the element of $\Gamma < \Mod_{g,n}^p$ corresponding to $\eta$ to be a multitwist.
\end{itemize}
By Lemma \ref{lemma:birmansimple}, this second property will
follow if we can show that $\eta = \mu^k$ for some $\mu \in \pi_1(\Sigma_{g,n}^{p-1},v)$ that can
be realized by a based simple closed curve.

Pick a point $q_S \in \delta_S$.  Let $\gamma_1^S$ be a path in $S$ from $v_S$ to $q_S$ and let
$\gamma_2^S$ be an embedded $q_S$-based loop in $S$ that goes once around $\delta_S$ in the positive direction.  
Define $\eta_S \in \pi_1(S,v_S)$ to be the homotopy class of the path 
$\gamma_1^S \cdot \gamma_2^S \cdot \overline{\gamma}_1^S$.  Clearly $\eta_S$ is freely homotopic to $\delta_S$.

Let $\eta = \rho_{\ast}(\eta_S)$ and $\delta = \rho(\delta_S)$.  If 
$\delta_S$ is a boundary component of $S$, then $\delta$ is a boundary
component of $\Sigma_{g,n}^{p-1}$.  If $\delta_S$ is freely homotopy to a puncture of $S$, then after possibly
modifying $\delta_S$ by a homotopy, we can assume that $\delta$ is a simple closed 
curve freely homotopic to a puncture of $\Sigma_{g,n}^{p-1}$.
The map $\rho|_{\delta_S} : \delta_S \rightarrow \delta$ is a $k$-fold cover for some $k \geq 1$.  
Let $q = \rho(q_S) \in \delta$ and $\gamma_1 = \rho_{\ast}(\gamma_1^S)$, so $\gamma_1$ is a path in 
$\Sigma_{g,n}^{p-1}$ from $v$ to $q$.  Also, let $\gamma_2$ be an embedded $q$-based loop in $\Sigma_{g,n}^{p-1}$
that goes once around $\delta$ in the positive direction.  By the above, we have that
$$\eta = \rho_{\ast}(\gamma_1^S) \cdot \rho_{\ast}(\gamma_2^S) \cdot \rho_{\ast}(\overline{\gamma}_1^S) = \gamma_1 \cdot \gamma_2^k \cdot \overline{\gamma}_1 = (\gamma_1 \cdot \gamma_2 \cdot \overline{\gamma}_1)^k.$$
Setting 
$$\mu = \gamma_1 \cdot \gamma_2 \cdot \overline{\gamma}_1 \in \pi_1(\Sigma_{g,n}^{p-1},v),$$
the curve $\mu$ is freely homotopic to the simple closed curve $\gamma_2$, so by Lemma
\ref{lemma:conjugatesimple} we get that $\mu$ can in fact be realized by a based simple closed
curve, as desired.

\begin{step}
The induced map $V_K \rightarrow \HH_1(\Gamma;\Q)$ is the zero map.
\end{step}

Since the conjugation action of $\Gamma$ on itself induces the trivial action on $\HH_1(\Gamma;\Q)$,
the map in question factors through $(V_K)_{\Gamma}$.  We will prove that
$(V_K)_{\Gamma}=0$.  Since we are assuming that Conjecture \ref{conjecture:actionconjecture} 
holds for $\Sigma_{g,n}^{p-1}$, we know that $\Gamma$ does not fix any nonzero vector in $V_K$, i.e.\ that
$(V_K)^{\Gamma}=0$.  
As in Step 1, let $S$ be the cover of $\Sigma_{g,n}^{p-1}$ corresponding to $K$.  Observe that
$V_K \cong \HH_1(S';\Q)$, where $S'$ is the closed surface obtained from $S$ 
by filling in all of its punctures and gluing discs to all of its
boundary components.  In particular, $V_K$ is naturally isomorphic to its dual via
the algebraic intersection pairing, so by 
Lemma \ref{lemma:selfdualcoinv} we can conclude that $(V_K)_{\Gamma} = 0$, as desired.
\end{proof}

\subsection{Everything is contained in nonseparating curve stabilizers}
\label{section:curvestab}

In this section, we prove Lemma \ref{lemma:curvestabsurject} below, which says that if $\Gamma < \Mod_{g,n}^p$
is a finite-index subgroup, then $\HH_1(\Gamma;\Q)$ is ``concentrated'' in the stabilizers of
nonseparating simple closed curves.  We need the following definition.

\begin{definition}
Let $\Gamma$ be a subgroup of $\Mod_{g,n}^p$.
\begin{itemize}
\item If $\gamma$ is a simple closed curve on $\Sigma_{g,n}^p$, then let
$$\Gamma_{\gamma} = \Set{$g \in \Gamma$}{$g$ fixes the isotopy class of $\gamma$}.$$
\item If $S$ is a subsurface of $\Sigma_{g,n}^p$, then let $\Gamma_S$ be the intersection
of $\Gamma$ with the image of the natural map $\Mod(S) \rightarrow \Mod_{g,n}^p$.
\end{itemize}
\end{definition}

\noindent
We can now state our lemma.

\begin{lemma}
\label{lemma:curvestabsurject}
Fix $g \geq 3$ and $n,p \geq 0$.  Assume that Conjecture \ref{conjecture:actionconjecture} holds for 
$\Sigma_{g-1,n+1}^p$.  Let $\Gamma < \Mod_{g,n}^p$ 
be a finite-index subgroup and let $\gamma$ be a nonseparating simple closed curve on $\Sigma_{g,n}^p$.
Then the natural map $\HH_1(\Gamma_{\gamma};\Q) \rightarrow \HH_1(\Gamma;\Q)$ is surjective.
\end{lemma}
\begin{proof}
Lemma \ref{lemma:transfer} implies that the vertical maps in the commutative diagram
$$\begin{CD}
\HH_1((\Gamma \cap \Mod_{g,n}^p(3))_{\gamma};\Q) @>>> \HH_1(\Gamma \cap \Mod_{g,n}^p(3);\Q) \\
@VVV                                             @VVV \\
\HH_1(\Gamma_{\gamma};\Q)                   @>>> \HH_1(\Gamma;\Q),
\end{CD}$$
are surjective.  It is therefore enough to show that the map 
$$\HH_1((\Gamma \cap \Mod_{g,n}^p(3))_{\gamma};\Q) \longrightarrow \HH_1(\Gamma \cap \Mod_{g,n}^p(3);\Q)$$
is surjective.  Replacing $\Gamma$ with $\Gamma \cap \Mod_{g,n}^p(3)$, we can thus assume without loss of
generality that $\Gamma < \Mod_{g,n}^p(3)$.  The proof will have four steps.  For the first two,
recall that $\Curves_{g,n}^p$ is the curve complex and $\CNosep_{g,n}^p$ is the nonseparating curve complex.

\begin{remark}
One could combine the proofs of Steps 1 and 2 by appealing to the fact that 
$\HH_i(\CNosep_{g,n}^p / \Gamma;\Q)=0$ for $1 \leq i \leq g-2$, 
which as we said after the statement of Theorem \ref{theorem:boggicurvecpx} can
be extracted from Boggi's work.  We have instead decided to appeal to Theorem \ref{theorem:boggicurvecpx},
which is explicitly proven in Boggi's paper.  This necessitates the short topological argument in Step 2.
\end{remark}

\BeginSteps
\begin{step}
The map
$$\bigoplus_{\gamma \in (\Curves_{g,n}^p)^{(0)}} \HH_1(\Gamma_{\gamma};\Q) \rightarrow \HH_1(\Gamma;\Q)$$
is a surjection.
\end{step}

Since $\Gamma < \Mod_{g,n}^p(3)$, the group $\Gamma$ acts on $\Curves_{g,n}^p$ without rotations.
Since $g \geq 3$, Theorems \ref{theorem:harercurvecpx} and \ref{theorem:boggicurvecpx} say that
$$\HH_1(\Curves_{g,n}^p;\Q) = \HH_1(\Curves_{g,n}^p/\Gamma;\Q) = 0.$$
The desired conclusion thus follows from Theorem \ref{theorem:equivariant}.

\Figure{figure:eliminateseparating}{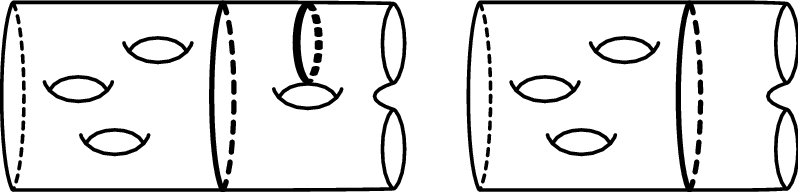}{a. The genus of $T_2$ is positive, so we can
find a nonseparating curve $\gamma$ in it. \CaptionSpace b. The genus of $T_2$ is $0$, so $T_1$
has genus $g \geq 3$ and strictly fewer than $n+p$ boundary components/punctures.}

\begin{step}
The map
\begin{equation}
\label{eqn:allnosep}
\bigoplus_{\gamma \in (\CNosep_{g,n}^p)^{(0)}} \HH_1(\Gamma_{\gamma};\Q) \rightarrow \HH_1(\Gamma;\Q)
\end{equation}
is a surjection.
\end{step}

Assume as an inductive hypothesis that the claim is true for all smaller values of $n$ and
$p$ (this assumption is vacuous for $n=p=0$).  As notation, for $\phi \in \Gamma$ we will write $[\phi]_{\Gamma}$
for the associated element of $\HH_1(\Gamma;\Q)$.
By Step 1, we must show that for every nontrivial separating curve $\delta$, the image of the map
$\HH_1(\Gamma_{\delta};\Q) \rightarrow \HH_1(\Gamma;\Q)$ is contained in the image of the map
\eqref{eqn:allnosep}.  Consider $\phi \in \Gamma_{\delta}$.  It is enough to show that
$[\phi^k]_{\Gamma}$ is in the image of the map in \eqref{eqn:allnosep} for some $k \geq 1$.

Assume that $\delta$ separates $\Sigma_{g,n}^p$ into subsurfaces $T_1$ and $T_2$.  Replacing
$\phi$ by $\phi^2$ if necessary, we can assume that $\phi$ does not exchange $T_1$
and $T_2$ (up to homotopy).  In other words, we can write $\phi = f_1 \cdot f_2$, 
where $f_i \in (\Mod_{g,n}^p)_{T_i}$ (we will say that $f_i$ is ``supported'' on $T_i$).  Since $\Gamma$ is a 
finite-index subgroup of $\Mod_{g,n}^p$, there exists some
$k \geq 1$ such that $f_i^k \in \Gamma$ for $i=1,2$.  
The mapping classes $f_1$ and $f_2$ commute, so $\phi^k = f_1^k \cdot f_2^k$.  It is therefore enough
to show that $[f_i^k]_{\Gamma}$ is in the image of \eqref{eqn:allnosep} for $i=1,2$.  

By symmetry, it is enough to show this for $i=1$.  There are two cases.
\begin{itemize}
\item The genus of $T_2$ is positive (see Figure \ref{figure:eliminateseparating}.a).  
In this case, we can find some simple closed nonseparating
curve $\gamma$ in $T_2$.  
Since $f_1$ is supported on $T_1$, the homotopy class
of $\gamma$ is fixed by $f_1^k$.  In other words, $f_1^k \in \Gamma_{\gamma}$, so $[f_1^k]_{\Gamma}$ is
in the image of \eqref{eqn:allnosep}.
\item The genus of $T_2$ is zero (see Figure \ref{figure:eliminateseparating}.b).  
In this case, the genus of $T_1$ is $g$, which is at least $3$.  Moreover,
since $\delta$ is not homotopic to a boundary component or puncture, the number of
punctures/boundary components of $T_1$ must be strictly less than $n+p$.  Letting $\CNosep(T_1)$
be the nonseparating curve complex of $T_1$, induction gives that the map
$$\bigoplus_{\gamma \in (\CNosep(T_1))^{(0)}} \HH_1((\Gamma_{T_1})_{\gamma};\Q) \rightarrow \HH_1(\Gamma_{T_1};\Q)$$
is surjective.  In particular, its image must contain the element of $\HH_1(\Gamma_{T_1};\Q)$ corresponding
to $f_1^k$.  This implies that the image of \eqref{eqn:allnosep} also must contain $[f_1^k]_{\Gamma}$.
\end{itemize}

\begin{step}
Let $S$ be an embedded subsurface of $\Sigma_{g,n}^p$ such that $S \cong \Sigma_{g-1,n+1}^p$.  
Let $\gamma$ be a nonseparating curve on $\Sigma_{g,n}^p$ that lies entirely within 
$\Sigma_{g,n}^p \setminus \Interior(S) \cong \Sigma_{1,1}$ (see Figure \ref{figure:stabilizerdiagram}.a).  
Then the natural map
$$\HHH_1(\Gamma_S;\Q) \longrightarrow \HHH_1(\Gamma_{\gamma};\Q)$$
is a surjection.
\end{step}

\Figure{figure:stabilizerdiagram}{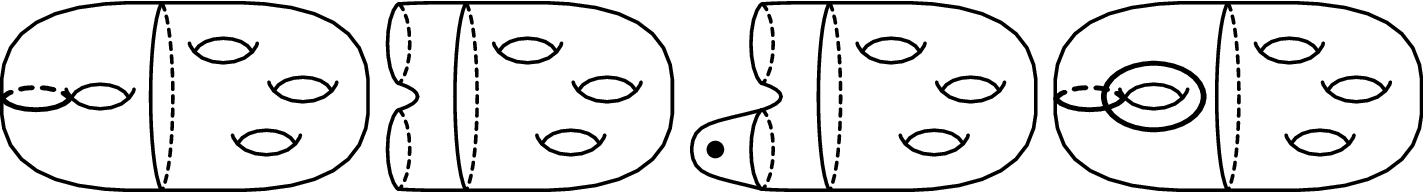}{a. $S$ is a subsurface of $\Sigma_{g,n}^p$ such 
that $S \cong \Sigma_{g-1,n+1}^p$ and $\gamma$ is a nonseparating curve on $\Sigma_{g,n}^p$ that lies entirely
within $\Sigma_{g,n}^p \setminus \Interior(S) \cong \Sigma_{1,1}$.
\CaptionSpace b. The complement $X$ of a regular neighborhood of $\gamma$.
\CaptionSpace c. The result $X'$ of gluing a punctured disc to $X$.
\CaptionSpace d. $S_i$ is the complement of a regular neighborhood of $\eta_i \cup \eta_{i+1}$.}

We make the following definitions.
\begin{itemize}
\item Let $X$ be the complement of a regular neighborhood of $\gamma \subset \Sigma_{g,n}^p$ (see
Figure \ref{figure:stabilizerdiagram}.b), so $X \cong \Sigma_{g-1,n+2}^p$.
\item Let $X'$ be the result of gluing a punctured disc to one of the two boundary components
of $X$ that are not boundary components of $\Sigma_{g,n}^p$ (see Figure \ref{figure:stabilizerdiagram}.c), 
so $X' \cong \Sigma_{g-1,n+1}^{p+1}$.
\item Let $X''$ be the result of filling in the puncture on the glued-on disc in $X'$, so
$X'' \cong \Sigma_{g-1,n+1}^p$.
\end{itemize}
We have an embedding $S \hookrightarrow \Sigma_{g,n}^p$.  As
shown in Figures \ref{figure:stabilizerdiagram}.a--c, we can arrange for there
to exist embeddings $S \hookrightarrow X$
and $S \hookrightarrow X'$ and $S \hookrightarrow X''$ such that the diagram
$$\xymatrix{
\Sigma_{g,n}^p & X \ar[l] \ar[r] & X' \ar[r] & X''\\
S \ar[u] \ar[ur] \ar[urr] \ar[urrr] & & &}$$
commutes.  There is an induced commutative diagram
$$\xymatrix{
(\Mod_{g,n}^p)_{\gamma} & \Mod(X) \ar[l] \ar[r] & \Mod(X') \ar[r] & \Mod(X'')\\
\Mod(S) \ar[u] \ar[ur] \ar[urr] \ar[urrr] & & &}$$
of mapping class groups.  Define $\Gamma_X \subset \Mod(X)$ to be the pullback
of $\Gamma_{\gamma} \subset (\Mod_{g,n}^p)_{\gamma}$ to $\Mod(X)$.  Also, define $\Gamma_X'$
to be the image of $\Gamma_X$ in $\Mod(X')$ and $\Gamma_X''$ to be the image of
$\Gamma_X'$ in $\Mod(X'')$.  We then have a commutative diagram
$$\xymatrix{
\Gamma_{\gamma} & \Gamma_X \ar[l] \ar[r] & \Gamma_X' \ar[r] & \Gamma_X''\\
\Gamma_S \ar[u] \ar[ur] \ar[urr] \ar[urrr] & & &}$$
Passing to first homology, we have a commutative diagram
$$\xymatrix{
\HHH_1(\Gamma_{\gamma};\Q) & \HHH_1(\Gamma_X;\Q) \ar[l] \ar[r] & \HHH_1(\Gamma_X';\Q) \ar[r] & \HHH_1(\Gamma_X'';\Q)\\
\HH_1(\Gamma_S;\Q) \ar[u] \ar[ur] \ar[urr] \ar[urrr] & & &}$$
Our goal is to show that the map $\HH_1(\Gamma_S;\Q) \rightarrow \HHH_1(\Gamma_{\gamma};\Q)$ is a surjection.
To do this, it is enough to show that all the maps on the first row of this commutative diagram are
isomorphisms and that the map $\HH_1(\Gamma_S;\Q) \rightarrow \HHH_1(\Gamma_X'';\Q)$ is a surjection.  We
deal with each of these claims in turn.
\BeginClaims
\begin{claim}
The map $\HHH_1(\Gamma_X;\Q) \rightarrow \HHH_1(\Gamma_{\gamma};\Q)$ is an isomorphism.
\end{claim}

Restricting exact sequence \eqref{eqn:complementscc} to $\Gamma_X$, we obtain
a short exact sequence
$$1 \longrightarrow \Z \longrightarrow \Gamma_X \longrightarrow \Gamma_{\gamma} \longrightarrow 1.$$
Here $\Z$ is generated by $(T_x T_y^{-1})^k = T_x^k T_y^{-k}$, where $x$ and $y$ are the boundary components
of $\Sigma_{g,n}^b \setminus \Interior(X)$ and $k \geq 1$.  Since the kernel $\Z$ is generated
by a multitwist, we conclude that the map $\HHH_1(\Gamma_X;\Q) \rightarrow \HHH_1(\Gamma_{\gamma};\Q)$ 
is an isomorphism.

\begin{claim}
The map $\HHH_1(\Gamma_{X};\Q) \rightarrow \HHH_1(\Gamma_X';\Q)$ is an isomorphism.
\end{claim}

Restricting exact sequence \eqref{eqn:boundarytopuncture} to $\Gamma_X$, we obtain
a short exact sequence
$$1 \longrightarrow \Z \longrightarrow \Gamma_X \longrightarrow \Gamma_X' \longrightarrow 1.$$
Here $\Z$ is generated by $T_{\beta}^l$, where $\beta$ is the boundary component of $X$ to
which we are gluing a punctured disc to obtain $X'$ and $l \geq 1$.  Since the kernel $\Z$ is generated
by a multitwist, we conclude that the map $\HHH_1(\Gamma_X;\Q) \rightarrow \HHH_1(\Gamma_{X}';\Q)$ 
is an isomorphism.

\begin{claim}
The map $\HHH_1(\Gamma_X';\Q) \rightarrow \HHH_1(\Gamma_X'';\Q)$ is an isomorphism.
\end{claim}

We have a Birman exact sequence
\begin{equation}
\label{eqn:modbirman}
1 \longrightarrow \pi_1(\Sigma_{g-1,n+1}^p) \longrightarrow \Mod(X') \longrightarrow \Mod(X'') \longrightarrow 1.
\end{equation}
Restricting this to $\Gamma_X'$, we get a short exact sequence
\begin{equation}
\label{eqn:gammaxbirman}
1 \longrightarrow \pi_1(\Sigma_{g-1,n+1}^p) \cap \Gamma_X' \longrightarrow \Gamma_X' \longrightarrow \Gamma_X'' \longrightarrow 1.
\end{equation}
Lemma \ref{lemma:fillpuncture} says that the map 
$\HH_1(\pi_1(\Sigma_{g-1,n+1}^p) \cap \Gamma_X';\Q) \rightarrow \HHH_1(\Gamma_X';\Q)$ is the zero map, so
we conclude that the map $\HHH_1(\Gamma_X';\Q) \rightarrow \HHH_1(\Gamma_X'';\Q)$ is an isomorphism.

\begin{claim}
The map $\HH_1(\Gamma_S;\Q) \rightarrow \HHH_1(\Gamma_X'';\Q)$ is a surjection.
\end{claim}

Since the map $\Mod(S) \rightarrow \Mod(X'')$ is an isomorphism, one might first think that
the map $\Gamma_S \rightarrow \Gamma_X''$ is an isomorphism.  However, this need not hold.  Indeed,
composing the isomorphism $\Mod(X'') \rightarrow \Mod(S)$ with the map $\Mod(S) \rightarrow \Mod(X')$
gives a splitting of exact sequence \eqref{eqn:modbirman}.  If the map $\Gamma_S \rightarrow \Gamma_X''$
were an isomorphism, then in a similar way we would obtain a splitting of exact sequence
\eqref{eqn:gammaxbirman}, which need not split.  However, chasing through the definitions, we
see that the image of $\Gamma_S$ in $\Gamma_X''$ is a finite-index subgroup, so Lemma \ref{lemma:transfer}
implies that the map $\HH_1(\Gamma_S;\Q) \rightarrow \HHH_1(\Gamma_X'';\Q)$ is a surjection.

\begin{step}
Let $\gamma$ be a nonseparating curve on $\Sigma_{g,n}^p$.  Then the natural map
$\HH_1(\Gamma_{\gamma};\Q) \rightarrow \HH_1(\Gamma;\Q)$ is a surjection.
\end{step}

Let $\gamma'$ be another nonseparating curve on $\Sigma_{g,n}^p$.  By Step 2, it is enough to show that the images
of the maps
\begin{equation}
\label{eqn:gammamaps}
\HH_1(\Gamma_{\gamma};\Q) \rightarrow \HH_1(\Gamma;\Q) \quad \text{and} \quad \HH_1(\Gamma_{\gamma'};\Q) \rightarrow \HH_1(\Gamma;\Q)
\end{equation}
are the same.  By Lemma \ref{lemma:connectcurves}, there is a sequence $\eta_1,\ldots,\eta_k$ of simple closed
nonseparating curves on $\Sigma_{g,n}^p$ such that $\gamma = \eta_1$, such that $\gamma' = \eta_k$, and such that
$\eta_i$ and $\eta_{i+1}$ intersect exactly 
once for $1 \leq i < k$.  To show that the images of the maps in \eqref{eqn:gammamaps} are the same,
it is thus enough to show that the images of the maps
\begin{equation}
\label{eqn:etamaps}
\HH_1(\Gamma_{\eta_i};\Q) \rightarrow \HH_1(\Gamma;\Q) \quad \text{and} \quad \HH_1(\Gamma_{\eta_{i+1}};\Q) \rightarrow \HH_1(\Gamma;\Q)
\end{equation}
are the same for $1 \leq i < k$.  Since $g \geq 3$, Corollary \ref{corollary:multitwistsvanish} implies that
these two maps factor through $\HHH_1(\Gamma_{\eta_i};\Q)$ and $\HHH_1(\Gamma_{\eta_{i+1}};\Q)$, respectively.
Let $S_i$ be the complement of a 
regular neighborhood of $\eta_i \cup \eta_{i+1}$ (see Figure \ref{figure:stabilizerdiagram}.d).  We thus have
$S_i \cong \Sigma_{g-1,n+1}^p$ and 
$\eta_i, \eta_{i+1} \subset \Sigma_{g,n}^p \setminus \Interior(S_i) \cong \Sigma_{1,1}$.
We have a commutative diagram
$$\begin{CD}
\Gamma_{S_i} @>>> \Gamma_{\eta_i} \\
@VVV              @VVV\\
\Gamma_{\eta_{i+1}} @>>> \Gamma
\end{CD}$$
Step 3 says that the maps $\HHH_1(\Gamma_{S_i};\Q) \rightarrow \HHH_1(\Gamma_{\eta_i};\Q)$ and
$\HHH_1(\Gamma_{S_i};\Q) \rightarrow \HHH_1(\Gamma_{\eta_{i+1}};\Q)$ are surjections, so we conclude
that the maps in \eqref{eqn:etamaps} have the same image, as desired.
\end{proof}

\subsection{The proof}
\label{section:theproof}

We finally are in a position to prove that Conjecture \ref{conjecture:actionconjecture} implies
Conjecture \ref{conjecture:ivanov}.

\begin{theorem}
\label{theorem:actionimpliesivanov}
Fix $g \geq 3$ and $n,p \geq 0$.  
Assume that Conjecture \ref{conjecture:actionconjecture} holds for $\Sigma_{g-1,n+1}^p$.  
Let $\Gamma < \Mod_{g,n}^p$ be a finite-index subgroup.  Then $\HH_1(\Gamma;\Q)=0$.
\end{theorem}
\begin{proof}
By Lemma \ref{lemma:transfer},
we can assume without loss of generality that $\Gamma$ is a normal subgroup of $\Mod_{g,n}^p$.
Lemma \ref{lemma:supertransfer} then implies that
$$\HH_1(\Mod_{g,n}^p;\Q) \cong (\HH_1(\Gamma;\Q))_{\Mod_{g,n}^p}.$$
Since $g \geq 3$, the group $\Mod_{g,n}^p$ is perfect, so $\HH_1(\Mod_{g,n}^p;\Q) = 0$.  It is thus
enough to show that $\Mod_{g,n}^p$ acts trivially on $\HH_1(\Gamma;\Q)$.  The group $\Mod_{g,n}^p$ 
is generated by Dehn twists $T_{\gamma}$ about nonseparating curves $\gamma$.  It is thus enough
to show that $T_{\gamma}$ acts trivially on $\HH_1(\Gamma;\Q)$ for $\gamma$ nonseparating.  Clearly
the conjugation action of $T_{\gamma}$ on $\Gamma$ restricts to the identity on the subgroup
$\Gamma_{\gamma} < \Gamma$, so the desired result is an
immediate consequence of Lemma \ref{lemma:curvestabsurject}, which says that the map
$\HH_1(\Gamma_{\gamma};\Q) \rightarrow \HH_1(\Gamma;\Q)$ is surjective for all nonseparating curves $\gamma$.
\end{proof}

\section{A (conditional) construction of finite-index subgroups of the mapping class group that surject onto $\Z$}
\label{section:construction}

We close the paper by proving if Conjecture \ref{conjecture:actionconjecture} is false,
then Conjecture \ref{conjecture:ivanov} is also false.  Assume, therefore, that
$g \geq 2$ and $n \geq 1$ and $p \geq 0$.  Let $K < \pi_1(\Sigma_{g,n}^p)$ be
a counterexample to Conjecture \ref{conjecture:actionconjecture}.  There thus exists
a nonzero vector $v_0 \in V_K$ such that the $\Mod_{g,n}^{p+1}$-orbit of $v_0$ is
finite.  Our goal is to find a finite-index subgroup $\Gamma < \Mod_{g,n}^{p+1}$
such that $\HH_1(\Gamma;\Q) \neq 0$.  

Consider the Birman exact sequence
$$1 \longrightarrow \pi_1(\Sigma_{g,n}^p) \longrightarrow \Mod_{g,n}^{p+1} \longrightarrow \Mod_{g,n}^p \longrightarrow 1.$$
Since $n \geq 1$, this exact sequence splits.  Fixing a splitting, we obtain an isomorphism
\begin{equation}
\label{eqn:semidirect}
\Mod_{g,n}^{p+1} \cong \pi_1(\Sigma_{g,n}^p) \rtimes \Mod_{g,n}^p.
\end{equation}
We remark that the action of $\Mod_{g,n}^p$ on $\pi_1(\Sigma_{g,n}^p)$ in \eqref{eqn:semidirect}
is {\em not} natural (it depends on the choice of splitting).

Let $\rho : \Mod_{g,n}^{p+1} \rightarrow \Aut(V_K)$ be the higher Prym representation.  Since
the orbit of $v_0 \in V_K$ is finite, we can find a finite-index subgroup $G' < \Mod_{g,n}^{p+1}$
such that $\rho(g)(v_0)=v_0$ for all $g \in G'$.  Regarding $\Mod_{g,n}^p$ as a subgroup
of $\Mod_{g,n}^{p+1}$ via the decomposition \eqref{eqn:semidirect}, let $G = G' \cap \Mod_{g,n}^p$.
Since $K < \pi_1(\Sigma_{g,n}^p)$ is characteristic, we can form the finite-index subgroup
$$\Gamma = K \rtimes G < \pi_1(\Sigma_{g,n}^p) \rtimes \Mod_{g,n}^p \cong \Mod_{g,n}^{p+1}.$$
From its semidirect product decomposition, we get that
$$\HH_1(\Gamma;\Q) \cong \HH_1(G;\Q) \oplus (\HH_1(K;\Q))_G.$$
To prove that $\HH_1(\Gamma;\Q) \neq 0$, it is thus enough to prove that $(\HH_1(K;\Q))_G \neq 0$.  Dually,
by Lemma \ref{lemma:selfdualcoinv} 
it is enough to construct a nonzero homomorphism $\psi : \HH_1(K;\Q) \rightarrow \Q$ which is invariant
under the natural $G$-action on $\Hom(\HH_1(K;\Q),\Q)$.

Recall that $V_K = \HH_1(K;\Q)/B$ where $B$ is the boundary subspace of $\HH_1(K;\Q)$.  It follows that
$V_K$ is the first rational homology group of the closed surface that results from taking the finite cover
of $\Sigma_{g,b}^p$ corresponding to $K < \pi_1(\Sigma_{g,b}^p)$ and gluing discs to all of its boundary
components and filling in all of its punctures.  In particular, $V_K$ is a symplectic vector space, i.e.\ it
has a nondegenerate alternating pairing $i : V_K \times V_K \rightarrow \Q$, namely the algebraic intersection
form.  Define $\psi : \HH_1(K;\Q) \rightarrow \Q$ to
be the composition
$$\HH_1(K;\Q) \longrightarrow V_K \stackrel{i(v_0,\cdot)}{\longrightarrow} \Q.$$
Since $i$ is nondegenerate and $v_0 \neq 0$, the map $\psi$ is nonzero.  Also, since $v_0$ is invariant
under $G$, it follows that $\psi$ is invariant under $G$, as desired.

\appendix
\section{Appendix : A counterexample in genus $1$}
\label{appendix:counterexample}

In this appendix, we sketch a counterexample to Conjecture \ref{conjecture:actionconjecture}
in genus $1$.  Let $Q_8 = \{\pm 1, \pm i, \pm j, \pm k\}$ be the $8$-element quaternion group.  We
have $i^2=-1$ and $ij = k$, so $Q_8$ is generated by $i$ and $j$.  It is well-known that $Q_8$
has the presentation
\begin{equation}
\label{eqn:qpres}
Q_8 = \langle \text{$i,j$ $|$ $i^4=1$, $i^2=j^2$, $j^{-1} i j = i^{-1}$} \rangle.
\end{equation}
There is a surjection $\rho : \pi_1(\Sigma_1^1) \rightarrow Q_8$ taking a free basis
of $\pi_1(\Sigma_1^1)$ to $i$ and $j$.  Let
$K = \ker(\rho)$.  By computing the action of the usual generators for $\Aut(F_2)$ on
the normal generators for $K$ given in \eqref{eqn:qpres}, one can check that
$K$ is characteristic in $\pi_1(\Sigma_1^1)$.  We
will prove that the action of $\Mod_1^2$ on $V_K$ has nontrivial finite orbits.  In fact,
it is a little easier to deal with a subgroup of $\Mod_1^2$.
Let $\Gamma \subset \Mod_1^2$ be the kernel of the natural map $\Mod_1^2 \rightarrow \Aut(Q_8)$.  Since
$\Gamma$ is a finite-index subgroup of $\Mod_1^2$, to prove
that the action of $\Mod_1^2$ on $V_K$ has nontrivial finite orbits, it is enough to show that
the action of $\Gamma$ on $V_K$ has nontrivial finite orbits.

\paragraph{Representation theory of the quaternion group.}
We first review the representation theory of $Q_8$.  The group $Q_8$ is an 
extension of its abelianization $(\Z/2)^2$ by its center $\Z/2$:
$$1 \longrightarrow \Z/2 \longrightarrow Q_8 \longrightarrow (\Z/2)^2 \longrightarrow 1.$$
Over $\Q$, the following
are all the irreducible representations of $Q_8$ (see \cite[Exercise 12.3]{SerreRep}).
\begin{itemize}
\item There are four $1$-dimensional irreducible representations that factor
through the abelianization $(\Z/2)^2$.
\item Regarding the rational quaternions as a $4$-dimensional vector space $\mathbb{H}_{\Q}$ over $\Q$,
the group $Q_8$ acts on $\mathbb{H}_{\Q}$ by left multiplication.  This makes $\mathbb{H}_{\Q}$ 
into an irreducible representation of $Q_8$.
\end{itemize}
By Schur's lemma, the endomorphism ring $\End(\mathbb{H}_{\Q})$ is a division ring over $\Q$.  In fact,
it is an easy exercise to see that $\End(\mathbb{H}_{\Q}) = \mathbb{H}_{\Q}$, where 
$\mathbb{H}_{\Q}$ acts on itself on the right.

\paragraph{Decomposing $V_K$.}
The cover of $\Sigma_1^1$ corresponding to $K$ is a genus $3$ surface with $4$ punctures.  By
definition, we have $V_K \cong \HH_1(\Sigma_3;\Q)$.  The group $Q_8$ acts on $V_K$, and we claim
that as a $Q_8$-representation we have
$V_K = \Q^2 \oplus \mathbb{H}_{\Q}$;
here $\Q^2$ is a two-dimensional trivial representation of $Q_8$.  First, there is a $Q_8$-equivariant
map $\rho : V_K \rightarrow \Q^2$ induced by the branched cover $\Sigma_3 \rightarrow \Sigma_1$
which comes from filling in the punctures in the cover $\Sigma_3^4 \rightarrow \Sigma_1^1$ corresponding
to $K$.  Lemma \ref{lemma:transfer} says that the map $\HH_1(\Sigma_3^4;\Q) \rightarrow \HH_1(\Sigma_1^1;\Q)$
is surjective, which implies that $\rho$ is surjective.  Hence $V_K = \Q^2 \oplus W$ 
for some $4$-dimensional representation
$W$ of $Q_8$.  Other than the identity, no finite-order orientation-preserving homeomorphism of a surface of genus at least $2$ can
act trivially on homology (see the proof of \cite[Theorem 6.8]{FarbMargalitBook}, where this is deduced
from the Lefschetz fixed point theorem).  This implies that the center of $Q_8$ acts nontrivially on $V_K$,
so the action of $Q_8$ on $W$ cannot factor through $(\Z/2)^2$.  We conclude that $W = \mathbb{H}_{\Q}$.

\paragraph{Finite orbits.}
The actions of $Q_8$ and $\Gamma$ on $V_K \cong \HH_1(\Sigma_3;\Q)$ come from homomorphisms 
$i : Q_8 \rightarrow \Sp_6(\Q)$ and $j : \Gamma \rightarrow \Sp_6(\Q)$ whose images lie in $\Sp_{6}(\Z)$.  
Let $G \subset \Sp_6(\Q)$ be the centralizer of $i(Q_8)$, so $G$ preserves the decomposition $V_K = \Q^2 \oplus \mathbb{H}_{\Q}$.  
The action of $\Gamma$ on $V_K$ commutes with the action of $Q_8$, so $j(\Gamma) \subset G \cap \Sp_{6}(\Z)$ (this is where
we use the fact that $\Gamma$ acts trivially on $Q_8$). 
We therefore get actions of $\Gamma$ on the subrepresentations $\Q^2$ and $\mathbb{H}_{\Q}$ of $V_K$.  

The action of $\Gamma$ on $\Q^2$ can be identified with the action on $\HH_1(\Sigma_1^1;\Q)$ arising
from the composition
$$\Gamma \hookrightarrow \Mod_1^2 \longrightarrow \Mod_1^1 \longrightarrow \Sp_{2}(\Q) = \SL_2(\Q);$$
here the second map arises from filling in the puncture of $\Sigma_1^2$ which corresponds
to the basepoint of $\pi_1(\Sigma_1^1)$.
The action of $\Gamma$ on the subrepresentation $\mathbb{H}_{\Q}$ of $V_K$ yields
a homomorphism $\psi : \Gamma \rightarrow \End(\mathbb{H}_{\Q}) = \mathbb{H}_{\Q}$.  Since
every symplectic matrix has determinant $1$, the image $\psi(\Gamma)$ lies in the group
of unit quaternions.  Also, $\psi$ fits into a commutative
diagram
$$\xymatrix{
\Gamma \ar[r]^-j \ar@/^3pc/[rr]^-{\psi} & G \cap \Sp_6(\Z) \ar[r] \ar@{^{(}->}[d] & \End(\mathbb{H}_{\Q}) = \mathbb{H}_{\Q} \\
                   & G \ar[ru] &}$$
This implies that $\psi(\Gamma)$ is a discrete
subgroup of the unit quaternions, so it must be finite.  In other words, the action
of $\Gamma$ on $\mathbb{H}_{\Q} \subset V_K$ factors through a finite group, so all of its orbits are finite.

\noindent
\begin{tabular*}{\linewidth}[t]{@{}p{\widthof{Department of Mathematics 253-37}+1in}@{}p{\linewidth - \widthof{Department of Mathematics 253-37}-1in}@{}}
{\raggedright
Andrew Putman\\
Department of Mathematics\\
Rice University, MS 136 \\
6100 Main St.\\
Houston, TX 77005\\
E-mail: {\tt andyp@rice.edu}}
&
{\raggedright
Ben Wieland\\
1729 Main St.\\
Providence, RI. 02906\\
E-mail: {\tt bwieland@gmail.com}}
\end{tabular*}

\end{document}